\documentclass{lms}

\newtheorem{theorem}{Theorem}[section] 
\newtheorem{lemma}[theorem]{Lemma}     

\newnumbered{assertion}{Assertion}    
\newnumbered{conjecture}{Conjecture}  
\newnumbered{definition}{Definition}
\newnumbered{hypothesis}{Hypothesis}
\newnumbered{remark}{Remark}
\newnumbered{note}{Note}
\newnumbered{observation}{Observation}
\newnumbered{problem}{Problem}
\newnumbered{question}{Question}
\newnumbered{algorithm}{Algorithm}
\newnumbered{example}{Example}
\newunnumbered{notation}{Notation} 

\title{On subsets of the normal rational curve} 

 \usepackage{amsmath}
\usepackage{amssymb}
\usepackage{graphicx}
\usepackage{amscd}

\author{Simeon Ball and Jan De Beule}

\classno{51E21, 15A03, 94B05}

\extraline{The first author acknowledges the support of the project MTM2014-54745-P of the Spanish {\em Ministerio de Econom\'ia y Competitividad.} The second author has been supported by a postdoctoral fellowship of the Research Foundation Flanders (Belgium) -- FWO, and by a research grant of the Research Foundation Flanders (Belgium) -- FWO, no. 1504514N.}

\begin{document}
\maketitle

\begin{abstract}
A normal rational curve of the $(k-1)$-dimensional projective space over ${\mathbb F}_q$ is an arc of size $q+1$,  since any $k$ points of the curve span the whole space.  In this article we will prove that if $q$ is odd then a subset of size $3k-6$ of a normal rational curve cannot be extended to an arc of size $q+2$. In fact, we prove something slightly stronger. Suppose that $q$ is odd and $E$ is a $(2k-3)$-subset of an arc $G$ of size $3k-6$. If $G$ projects to a subset of a conic from every $(k-3)$-subset of $E$ then $G$ cannot be extended to an arc of size $q+2$. Stated in terms of error-correcting codes we prove that a $k$-dimensional linear maximum distance separable code of length $3k-6$ over a field ${\mathbb F}_q$ of odd characteristic, which can be extended to a Reed-Solomon code of length $q+1$, cannot be extended to a linear maximum distance separable code of length $q+2$. 
\end{abstract}

\section{Introduction}

Let $\mathrm{V}_k({\mathbb F}_q)$ denote the $k$-dimensional vector space over ${\mathbb F}_q$, the finite field with $q$ elements.

Let $\mathrm{PG}_{k-1}({\mathbb F}_q)$ denote the $(k-1)$-dimensional projective space over ${\mathbb F}_q$.

An arc $S$ is a set of vectors of $\mathrm{V}_k({\mathbb F}_q)$ in which every subset of $S$ of size $k$ is a basis of the space, i.e. every $k$-subset is a set of linearly independent vectors. 
Equivalently, an arc of $\mathrm{PG}_{k-1}({\mathbb F}_q)$ is a set of points in which every subset of size $k$ spans the whole space.

The set of columns of a generator matrix of a $k$-dimensional linear maximum distance separable (MDS) code over ${\mathbb F}_q$ is an arc of $\mathrm{V}_k({\mathbb F}_q)$ and vice-versa, so arcs and linear MDS codes are equivalent objects. As in coding theory, we define the {\em weight} of a vector to be the number of non-zero coordinates that it has.

A normal rational curve is a set of $q+1$ vectors of $\mathrm{V}_k({\mathbb F}_q)$ 
$$
S=\{ (1,t,t^2,\ldots,t^{k-1}) \ | \ t \in {\mathbb F}_q \} \cup \{ (0,\ldots,0,1) \},
$$
or equivalently a set of $q+1$ points of $\mathrm{PG}_{k-1}({\mathbb F}_q)$. 

It is easy to see that a normal rational curve is an arc, since taking any $k$ elements of $S$ and we can form a $k \times k$ Vandermonde matrix whose determinant is non-zero. Thus, any $k$ vectors of $S$ are linearly independent.

In 1986, Seroussi and Roth \cite{RS1986} proved that if $4 \leqslant k \leqslant (q+3)/2$ then a normal rational curve cannot be extended to an arc of size $q+2$. In 1992, Storme \cite{Storme1992} extended this result to $4 \leqslant k \leqslant q+2-6\sqrt{q\ln q}$. In this article we will prove that if $q$ is odd then a subset of size $3k-6$ of a normal rational curve cannot be extended to an arc of size $q+2$.

Lemma~\ref{project} and Lemma~\ref{projecttoplane} follow from \cite[Theorem 27.5.1 and Lemma 27.5.2]{HT1991}. We include a proof for the sake of completeness.

\begin{lemma} \label{project}
The projection of a normal rational curve of $\mathrm{PG}_{k-1}({\mathbb F}_q)$ from any vector of the normal rational curve is contained in a normal rational curve of $\mathrm{PG}_{k-2}({\mathbb F}_q)$.
\end{lemma}

\begin{proof}
Suppose we wish to project 
$$
S=\{ (1,t,\ldots,t^{k-1}) \ | \ t \in {\mathbb F}_q \} \cup \{ (0,\ldots,0,1) \},
$$
from the point
$$
x=(1,s,s^2,\ldots,s^{k-1}).
$$
There is a change of basis matrix of $\mathrm{V}_k({\mathbb F}_q)$ that maps
$$
(1,t,t^2,\ldots,t^{k-1})  \mapsto ((ct+d)^{k-1},(ct+d)^{k-2}(at+b),\ldots,(ct+d)(at+b)^{k-2},(at+b)^{k-1}), 
$$
where $ad \neq bc$. Hence, we can find a change of basis that
changes the coordinates of $x$ to $(0,\ldots,0,1)$ and fixes $S$ set-wise. Projecting from $x$ is equivalent to deleting the last coordinate, so the projection of $S$ is contained in a normal rational curve of $\mathrm{PG}_{k-2}({\mathbb F}_q)$.

\end{proof}

\begin{lemma} \label{projecttoplane}
The projection of a normal rational curve of $\mathrm{PG}_{k-1}({\mathbb F}_q)$ from any $k-3$ vectors of the normal rational curve is contained in a conic of $\mathrm{PG}_{2}({\mathbb F}_q)$.
\end{lemma}

\begin{proof}
Let $D=\{ x_1,\ldots,x_{k-3} \}$ be $k-3$ vectors of the normal rational curve. 
By Lemma~\ref{project}, the projection of a normal rational curve of $\mathrm{PG}_{k-1}({\mathbb F}_q)$ from $x_1$ is contained in a normal rational curve of $\mathrm{PG}_{k-2}({\mathbb F}_q)$. Projecting this projection from the projection of $x_2$ we obtain a set contained in a normal rational curve of $\mathrm{PG}_{k-3}({\mathbb F}_q)$. Continuing in this way we see that the projection from $D$ of the normal rational curve is contained in a normal rational curve (i.e. a conic) of $\mathrm{PG}_{2}({\mathbb F}_q)$.
\end{proof}

The aim of this article is to prove the following theorem.

\begin{theorem} \label{main}
Suppose $q$ is odd and $G$ is an arc of $\mathrm{PG}_{k-1}({\mathbb F}_q)$ of size $3k-6$. Suppose that $E$ is a subset of $G$ of size $2k-3$ and that $G$ projects to a subset of a conic from every $(k-3)$-subset of $E$.  Then $G$ cannot be extended to an arc of size $q+2$.
\end{theorem}

The following theorem follows immediately from Lemma~\ref{projecttoplane} and Theorem~\ref{main}.

\begin{theorem} \label{nrcmain}
If $q$ is odd then a subset of $3k-6$ points on a normal rational curve of $\mathrm{PG}_{k-1}({\mathbb F}_q)$ cannot be extended to an arc of size $q+2$.
\end{theorem}

Stated in terms of error-correcting codes, Theorem~\ref{nrcmain} says the following.

\begin{theorem}
A $k$-dimensional linear maximum distance separable code of length $3k-6$ over a field ${\mathbb F}_q$ of odd characteristic, which can be extended to a Reed-Solomon code of length $q+1$, cannot be extended to a linear maximum distance separable code of length $q+2$. 
\end{theorem}

Note that some authors refer to the Reed-Solomon code as the cyclic code of length $q-1$, obtained from the normal rational curve by deleting the columns $(1,0,\ldots,0)$ and $(0,\ldots,0,1)$ in the generator matrix, and the code of length $q+1$ as the doubly-extended Reed-Solomon code.

\section{A set of equations associated with an arc}

Let $\det(v_1,\ldots,v_k)$ denote the determinant of the matrix whose $i$-th row is $v_i$, a vector of ${\mathrm V}_k({\mathbb F}_q)$. If $C=\{p_1,\ldots,p_{k-1} \}$ is an ordered set of $k-1$ vectors then we write
$$
\det(u,C)=\det(u,p_1,\ldots,p_{k-1}),
$$
where we evaluate the determinant with respect to a fixed canonical basis. 

Throughout $S$ will be an arbitrarily ordered arc of size $q+k-1-t$ of $\mathrm{V}_k({\mathbb F}_q)$.

Let $C$ be subset of $S$ of size $k-1$. There is a non-zero element $\alpha_C \in {\mathbb F}_q$, dependent on the first $k-2$ elements of $S$ and the ordering of $S$, such that the following lemma holds, see \cite[Lemma 7.20]{Ball2015} or \cite[Lemma 17]{Ball2016a}.

\begin{lemma} \label{theeqn}
Let $E$ be a subset of $S$ of size $k+t$. For any subset $A$ of $E$ of size $k-2$,
$$
\sum \alpha_C \prod_{z \in E \setminus C} \det(z,C)^{-1}=0,
$$
where the sum runs over the subsets $C$ of $E$ of size $k-1$ containing $A$.
\end{lemma}

In this article we will use the following set of equations, which are deduced from the equations in Lemma~\ref{theeqn}.

\begin{lemma} \label{theeqnodd}
Suppose that $q$ is odd. Let $E$ be a subset of $S$ of size $k+t-1$ and let $e \in S \setminus E$. For any subset $D$ of $E$ of size $k-3$,
$$
\sum \alpha_C \prod_{z \in (E \cup \{ e \}) \setminus C} \det(z,C)^{-1}=0,
$$
where the sum runs over the subsets $C$ of $E$ of size $k-1$ containing $D$.
\end{lemma}

\begin{proof}
Let us call the equation in Lemma~\ref{theeqn}, $\mathrm{eq}(A)$. Then for any $D$ which is a subset of $E$ of size $k-3$, consider
$$
-\mathrm{eq}(D \cup \{ e \})+\sum_{a \in E \setminus (D \cup \{ e \})} \mathrm{eq}(D \cup \{ a \}).
$$
The terms for which $e \in C$ cancel and the terms for which $e \not\in C$ appear twice. Since $q$ is odd we obtain the equation stated in the lemma.
\end{proof}

Since $S$ is an arc, every $(k-1)$-subset $C=\{p_1,\ldots,p_{k-1}\}$ of $S$ spans a hyperplane. In the dual space this hyperplane is a vector. To explicitly work with this vector we set up a duality between $\mathrm{V}_k({\mathbb F}_q)$ and its dual space. Let
$$
x_j=(-1)^{j+1} \det(p_1,\ldots,p_{k-1}),
$$
where the $j$-th coordinate of $p_1,\ldots,p_{k-1}$ has been deleted.

This allows us to write
$$
\det(u,C)=u\cdot x,
$$
where $x=(x_1,\ldots,x_k)$ and $\cdot$ is the standard scalar product, which by abusing notation slightly, we will write as $u \cdot C$.

We will use the word {\em point} for a $1$-dimensional subspace of ${\mathrm V}_k({\mathbb F}_q)$ and {\em line} for a $2$-dimensional subspace of ${\mathrm V}_k({\mathbb F}_q)$ and likewise for the dual space. Thus, in the dual space, the subspace spanned by a set $C$ of $(k-1)$ vectors of $S$ is a point and the subspace spanned by a set $A$ of $k-2$ vectors of $S$ is a line. Again, we will abuse notation slightly and refer to $C$ as a point of the dual space and $A$ as a line of the dual space.

\begin{lemma} \label{notcrossbasis}
Suppose that $A$ and $U$ are disjoint subsets of an arc of ${\mathrm V}_k({\mathbb F}_q)$ of sizes $k-2$ and $n+1$ respectively. If
$$
\psi (X)= \sum_{w \in U} \lambda_w \prod_{u \in U \setminus \{ w\}} (u\cdot X),
$$
for some $\lambda_w \in {\mathbb F}_q$, is zero at $n+1$ points on the line $A$, then $\psi \equiv 0$.
\end{lemma}

\begin{proof}
The polynomial $\psi$ is of degree $n$. By choosing a basis which includes the elements of $A$, we see that its restriction to the line $A$ is a homogeneous polynomial in two variables. If it is zero at $n+1$ distinct points then it is identically zero on the line $A$.

Let $w_0 \in U$. Then
$$
0=\psi(A \cup \{ w_0 \})=\lambda_{w_0} \prod_{u \in U \setminus \{ w_0\}} \det(u,A \cup \{ w_0\}),
$$ 
and so $\lambda_{w_0}=0$. Note that all the determinants are non-zero since $A \cup U$ is an arc.
\end{proof}

\section{Arcs in spaces of odd characteristic}

We will suppose from now on that $q$ is odd and $k\geqslant 5$. Note that Theorem~\ref{main} holds for $k=2$, $3$ and $4$ since there are no arcs of size $q+2$ in these spaces when $q$ is odd, see for example \cite{HT1991} or \cite{HS2001}. 

Let $n$ be a non-negative integer such that $n \leqslant |S|-k-t$.

Let $G$ be a subset of $S$ of size $k+t+n$, which will remain fixed and let $E$ be a subset of $G$ of size $k+t-1$, which will also remain fixed. Let $U=G \setminus E$.
Let $A$ be a subset of $E$ of size $k-2$ which for the most part will also remain fixed.

We define a matrix $\mathrm{P}_n$ whose rows are indexed by the $(k-1)$-subsets of $E$ which contain a $(k-3)$-subset of $A$. The columns of $\mathrm{P}_n$ are indexed by pairs $(D,w)$, where $D$ is a subset of $A$ of size $k-3$ and $w \in U$. The $(C,(D,w))$ entry of $\mathrm{P}_n$ is
$$
\prod_{u \in U \setminus \{ w \}} \det(u,C),
$$
if $C$ contains $D$ and zero otherwise.

\begin{lemma} \label{nowone}
There is no vector of weight one in the column space of $\mathrm{P}_n$.
\end{lemma}

\begin{proof}
Let $v$ be the row vector whose coordinates are indexed by the $(k-1)$-subsets $C$ of $E$ which contain a $(k-3)$-subset of $A$ and whose $C$ entry is
$$
\alpha_C \prod_{z \in G \setminus C} \det(z,C)^{-1}.
$$
Note that all the coordinates in $v$ are non-zero.

The standard scalar product of $v$ with the $(D,w)$ column of $\mathrm{P}_n$ is
$$
\sum \alpha_C \prod_{z \in G \setminus C} \det(z,C)^{-1}\prod_{u \in U \setminus \{ w \}} \det(u,C)=\sum \alpha_C \prod_{z \in (E \cup \{ w \}) \setminus C} \det(z,C)^{-1},
$$
where the sum runs over the subsets $C$ of $E$ of size $k-1$ containing $D$.
By Lemma~\ref{theeqnodd}, this sum is zero. 

Therefore, if there is a vector $r$ of weight one in the column space of $\mathrm{P}_n$ then $v \cdot r=0$. This implies one of the coordinates of $v$ is zero, which it is not.\end{proof}

Since $D$ is a set of $k-3$ linearly independent vectors, the subspace $\langle D \rangle$ has dimension $k-3$ and the quotient space $\mathrm{V}_k({\mathbb F}_q)/\langle D \rangle$ has dimension $3$. Let 
$$
G /D=\{ g+\langle D \rangle \ | \ g \in G \setminus D \}
$$
denote the set of $t+n+3$ vectors in this quotient space obtained from the vectors of $G \setminus D$.

Let $e$ be a fixed element of $E \setminus A$.

Let $\mathrm{M}_D$ be the matrix whose rows are indexed by $(k-1)$-subsets $C=D \cup L$, where $L$ is a  $2$-subset of $E\setminus A$ and whose $n+1$ columns are indexed by $w \in U$ and whose $(C,w)$ entry is 
$$
\prod_{u \in U \setminus \{ w \}} \det(u,C)=\prod_{u \in U \setminus \{ w \}} (u \cdot  C).
$$
Observe that $\mathrm{M}_D$ is a submatrix of $\mathrm{P}_n$.

\begin{lemma} \label{matrixmd}
If $G$ projects to a subset of a conic from $D$ then the row space of $\mathrm{M}_D$ is spanned by the rows indexed by $C=D\cup \{e,b \}$, where $b \in E \setminus (A\cup \{ e \})$.
\end{lemma}

\begin{proof}
Since $G$ projects to a subset of a conic from $D$, we have that $G/D$ is contained in the set of zeros of a homogeneous polynomial $f_D$ in 3 variables and of degree $2$.

Let $W=\{w_1,w_2,w_3 \}$ be a subset of $U$ of size $3$. Consider the $3\times 3$ submatrix $\mathrm{M}$ of $\mathrm{M}_D$ whose columns are indexed by $w \in W$, and whose rows are indexed by $D \cup \{e,a \}$, $D \cup \{e,b \}$ and $D \cup \{a, b \}$.
Define $g_D$ to be the polynomial of degree two, which is the determinant of $\mathrm{M}$ where we replace $w_1$ by $X$. The $C$-row in this determinant is
$$
(\prod_{u \in U \setminus \{ w_1 \}} (u \cdot C),\ \ \ (X\cdot C) \prod_{u \in U \setminus \{ w_1,w_2 \}} (u\cdot C),\ \ \ (X \cdot C) \prod_{u \in U \setminus \{ w_1,w_3 \}} (u\cdot C) ).
$$
By changing the basis so that the basis includes the elements of $D$, $g_D(X)$ will be a homogeneous polynomial in three variables. 
Clearly $g_D(w_2)=g_D(w_3)=0$ since the determinant will have repeated columns. Moreover $g_D(a)=0$, since the determinant will have a $2 \times 2$ submatrix of zeros. Similarly, 
$g_D(b)=0=g_D(e)=0$. Thus $g_D$ is the unique polynomial, up to scalar factor, whose set of zeros contains these five vectors of the quotient space. Hence, $g_D$ is a scalar multiple of $f_D$. Therefore, $g_D(w_1)$ is also zero and the determinant of $\mathrm{M}$ is zero.

Since every $3\times 3$ submatrix $\mathrm{M}$ of $\mathrm{M}_D$ has rank two, the $3 \times (n+1)$ matrix whose columns are indexed by $w \in U$, and whose rows are indexed by $D \cup \{e,a \}$, $D \cup \{e,b \}$ and $D \cup \{a, b \}$ has rank two. Therefore, in $\mathrm{M}_D$, the row indexed by $D \cup \{a, b \}$ must be a linear combination of the rows indexed by $D \cup \{e,a \}$ and $D \cup \{e,b \}$.
\end{proof}

\begin{lemma} \label{thepsis}
If $G$ projects to a subset of a conic from $D$ and $n\geqslant t$ then there exist $\lambda_w \in {\mathbb F}_q$ such that
$$
\psi_D(X)=\sum_{w \in U} \lambda_w \prod_{u \in U \setminus \{ w \}} (u\cdot X) \not\equiv 0,
$$
is zero at $D \cup L$, for all $2$-subsets $L$ of $E \setminus A$.
\end{lemma}

\begin{proof}
By Lemma~\ref{matrixmd}, the matrix $\mathrm{M}_D$ has rank at most $t \leqslant n$. Since $\mathrm{M}_D$ has $n+1$ columns there are elements $\lambda_w \in {\mathbb F}_q$, not all zero, such that 
$$
\sum_{w \in U} \lambda_w v_w=0,
$$
where $v_w$ is the column of $\mathrm{M}_D$ indexed by $w$. Since it is zero, the $C$ coordinate of this linear combination is zero. However, it is also the evaluation of $\psi_D$ at $C$. Thus, we have that $\psi_D(C)=0$ for all $C=D \cup L$, where $L$ is a $2$-subset of $E \setminus A$. 
\end{proof}

In light of Lemma~\ref{thepsis}, we will now take $n=t$.

Let 
$$
v_D=\sum_{w \in U} \lambda_w v_{D,w},
$$
where $v_{D,w}$ is the column of $\mathrm{P}_n$ indexed by $(D,w)$. Note that the $C$-coordinate of $v_D$ is the evaluation of $\psi_D$ if $C \supset D$ and zero otherwise.

By Lemma~\ref{thepsis}, the $C$-coordinate in $v_D$ is zero if $C=D \cup L$, for some $2$-subset $L$ of $E \setminus A$, so the only possibly non-zero coordinates of $v_D$ are indexed by $(k-1)$-subsets $C=A \cup \{b\}$, for some $b \in E \setminus A$.

Let us define a $(t+1) \times (k-2)$ matrix $\mathrm{Q}_t$, whose rows are indexed by $(k-1)$-subsets $C=A \cup \{ b\}$, for some $b \in E \setminus A$ and whose columns are the restriction of $v_D$ to these $C$ coordinates. So the columns are indexed by the $(k-3)$-subsets $D$ of $A$. The following lemmas are immediate from this discussion and the definition of $\mathrm{Q}_t$.

\begin{lemma} \label{evalpsi}
The $C$ entry of the $D$ column of $\mathrm{Q}_t$ is $\psi_D(C)$.
\end{lemma}

\begin{lemma} \label{weightoneQ}
If there is a vector of weight one in the column space of $\mathrm{Q}_t$ then there is a vector of weight one in the column space of $\mathrm{P}_t$.
\end{lemma}

\section{Arcs that extend to arcs of size $q+2$}

Suppose now that $n=t=k-3$.

Thus $G$ is an arc of size $3k-6$ which extends to an arc of size $q+2$, $E$ is a subset of $G$ of size $2k-4$, $A$ is a subset of $E$ of size $k-2$ and $D$ is a subset of $A$ of size $k-3$.

\begin{lemma} \label{projpsi}
If $G$ projects to a subset of a conic from all $(k-3)$-subsets of $E$ then $\psi_D(C)=0$ for all $C \subseteq E \setminus (A \setminus D)$. Moreover, $\psi_D(E \setminus D) \neq 0$. 
\end{lemma}

\begin{proof}
Let $D$ and $D'$ be $(k-3)$-subsets of $E \setminus \{ a \}$, where $\{ a \}=A \setminus D$. Suppose that $|D' \cap D|=k-4$.

By hypothesis, $G$ projects to a subset of a conic from $D$ and $D'$. By Lemma~\ref{thepsis}, with $A=D \cup \{ a \}$ and $A=D' \cup \{ a \}$ respectively, there is a $\psi_D$ and $\psi_{D'}$ which are both zero at $D \cup D' \cup \{ b\}$, for all $b \in E \setminus (D \cup D' \cup \{ a \})$.

There are $n+1=t+1$ coefficients $\lambda_w$ in the definition of $\psi_D$ and there are $t$ elements in $b \in E \setminus (D \cup D' \cup \{ a \})$. So up to scalar multiple, $\psi_D$ is determined by the equations $\psi( D \cup D' \cup \{ b\})=0$. Note that for each $b$, the equation $\psi( D \cup D' \cup \{ b\})=0$ gives an independent condition, since for each $b$ this is a distinct point on the line $D \cup D'$.

However, $\psi_{D'}$ is determined by the same set of equations, so we conclude that $\psi_D$ and $\psi_{D'}$ are scalar multiples of each other.

Since $\psi_{D'}(D' \cup L')=0$, for all $2$-subsets $L'$ of $E \setminus(D' \cup \{ a \})$, we have that $\psi_D(D' \cup L')=0$. 

Now, repeating the above with $D$ replaced with $D'$ and $D'$ replaced with $D''$, where $|D' \cap D''|=k-4$ and $|D \cap D''|=k-5$, we have that $\psi_D(D'' \cup L'')=0$ for all $2$-subsets $L''$ of $E \setminus(D'' \cup \{ a \})$. Continuing in this way we conclude that $\psi_D(C)=0$ for all $(k-1)$-subsets $C$ of $E \setminus \{ a \}$.


If $\psi_D(E \setminus D)=0$ then $\psi_D$ has $t+1$ zeros on the line $E \setminus A$. By Lemma~\ref{notcrossbasis}, this implies that $\psi_D \equiv 0$, which it is not by Lemma~\ref{thepsis}. Therefore,  $\psi_D(E \setminus D) \neq 0$, which completes the proof.

\end{proof}

\begin{lemma} \label{woneQ}
If $n=t=k-3$ and $G$ projects to a subset of a conic from every $(k-3)$-subset $D$ of $E$ then there is a vector of weight one in the column space of $\mathrm{Q}_t$.
\end{lemma}

\begin{proof}
Since $t=k-3$ the matrix $\mathrm{Q}_t$ is a square matrix. By Lemma~\ref{evalpsi}, a column of $\mathrm{Q}_t$ is the evaluation of 
$\psi_D(X)$
at the $(k-1)$-subsets $C$ of $A \cup \{ b \}$, where $b \in E \setminus A$. If $\mathrm{Q}_t$ does not have full rank then there is a linear combination of the columns which is zero. The $C$ coordinate of a linear combination of the columns is the evaluation of 
$$
\psi_A(X)=\sum \mu_D \psi_D(X),
$$
for some $\mu_D \in {\mathbb F}_q$, not all zero, where the sum runs over the $(k-3)$-subsets $D$ of $A$.

If $\psi_A$ is zero at all points $A \cup \{ b \}$, where $b \in E \setminus A$ then it is zero at $t+1$ points on the line $A$. By Lemma~\ref{notcrossbasis}, $\psi_A \equiv 0$. However, for a subset $D$ of $E$ of size $k-3$,
$$
\psi_A(E \setminus D)= \mu_D \psi_D(E \setminus D),
$$
and by Lemma~\ref{projpsi}, $\psi_D(E \setminus D) \neq 0$. Therefore $\mu_D=0$, which implies that the columns of $\mathrm{Q}_t$ are linearly independent. Since $\mathrm{Q}_t$ is a square matrix there is a vector of weight one in the column space of $\mathrm{Q}_t$.
\end{proof}

We can now prove Theorem~\ref{main}.

\begin{proof} (of Theorem~\ref{main}) 
Let $G$ be an arc of ${\mathrm V}_k({\mathbb F}_q)$ of size $3k-6$ and suppose that $E$ is a subset of $G$ of size $2k-3$ and that $G$ projects to a subset of a conic from every $(k-3)$-subset of $E$. Suppose that $G$ extends to an arc of size $S$ of size $q+2$, so $t=k-3$.

By Lemma~\ref{woneQ}, there is a vector of weight one in the column space of $\mathrm{Q}_{k-3}$. By Lemma~\ref{weightoneQ}, there is a vector of weight one in the column space of $\mathrm{P}_{k-3}$, which contradicts Lemma~\ref{nowone}.
\end{proof}

\section{Comments}

It is a long-standing conjecture, dating back to the 1950's, that there are no arcs of size $q+2$ for $4 \leqslant k \leqslant q-2$, see \cite{HS2001}, \cite{MS1977} or \cite{Vardy2006} for example. It was proven in \cite{Ball2012} that the conjecture is true for $q$ prime. 

This article can be considered as an example of how the system of equations in Lemma~\ref{theeqn} can be used to prove results about arcs and try to verify this conjecture when $q$ is not a prime. It is by no means easy, but this article at least demonstrates that it is possible. 

With regard to the result itself, the hypothesis that the projections lie on a conic may not be necessary. Indeed, computations of explicit examples indicate that this is in fact the worst-case scenario. In other words, it is only in this case that have to take $n$ so large to prove that $\mathrm{P}_{n}$ has a vector of weight one in its column space, assuming that $G$ extends to a $(q+2)$-arc. For other arcs, it appears that $\mathrm{P}_{n}$ has a vector of weight one in its column space for smaller $n$. However, $\mathrm{Q}_{n}$ does not have a vector of weight one in its column space, if we remove the projection hypothesis, so one must consider the larger matrix $\mathrm{P}_{n}$ in this case. 

In \cite{Chowdhury2015} it is conjectured that a larger matrix $\mathrm{M}_{n}$ has full rank and in \cite{Ball2016a} that it has a vector of weight one in its column space, for any arc where $k \leqslant p+n(p-2)$ and where $p$ is the characteristic of ${\mathbb F}_q$. Again, we are assuming that $G$ extends to a $(q+2)$-arc. This would imply that there are no arcs of size $q+2$ for $k \leqslant (pq-2q+6p-10)/(2p-3)$. We conjecture that the same is true for the smaller matrix $\mathrm{P}_{n}$. In other words we conjecture that $\mathrm{P}_{n}$ has a vector of weight one in its column space, for any arc that extends to a $(q+2)$-arc, when $k \leqslant p+n(p-2)$, which would contradict Lemma~\ref{nowone} and imply that there are no arcs of size $q+2$ for $k \leqslant (pq-2q+6p-10)/(2p-3)$.

\affiliationone{
   Simeon Ball\\
   Departament de Matem\`atiques, \\
Universitat Polit\`ecnica de Catalunya, \\
M\`odul C3, Campus Nord,\\
c/ Jordi Girona 1-3,\\
08034 Barcelona, Spain \\
   \email{simeon@ma4.upc.edu}}
   \affiliationtwo{
   Jan De Beule\\
Vakgroep Wiskunde,\\
Vrije Universiteit Brussel,\\
Pleinlaan 2,\\
B-1050 Brussels, \\
Belgium \\
   \email{jan@debeule.eu}}

\end{document}